\documentclass[12pt,letterpaper]{amsart}

\usepackage[margin=1in]{geometry}
\usepackage{graphicx}
\usepackage{amsthm, amscd}
\usepackage{enumerate}

\usepackage{hyperref}

\newcommand{\C}{\mathbb{C}}
\newcommand{\N}{\mathbb{N}}
\newcommand{\R}{\mathbb{R}}
\newcommand{\Z}{\mathbb{Z}}

\DeclareMathOperator{\spn}{span}
\DeclareMathOperator{\chr}{char}
\newcommand{\ket}[1]{| #1\rangle}

\theoremstyle{plain}
\newtheorem{theorem}{Theorem}
\newtheorem{claim}{Claim}
\newtheorem{conjecture}{Conjecture}
\newtheorem{prop}{Proposition}

\theoremstyle{definition}
\newtheorem{definition}{Definition}

\theoremstyle{remark}
\newtheorem{remark}{Remark}
\newtheorem*{acknowledgements}{Acknowledgements}

\begin{document}
\title{Trigonometry through Hopf algebras}
\author{Aaron Brookner}
\thanks{}
\subjclass[2010]{} 
\date{\today}

\begin{abstract}
We consider various coalgebras and Hopf algebras related to trigonometry in this paper. 
In section \ref{sect:trigco} we answer the question of whether the trigonometric coalgebra 
can be defined to have the structure of a Hopf algebra, and over which fields. In sections 
\ref{sect:units} and \ref{sect:FFT}, we generalize this construction to $p$-dimensional 
coalgebras, where $p$ is a prime number, and show how this generalization is related to 
the finite Fourier transform. Then we consider the structure of the Hopf algebra of 
polynomial functions on the circle $S^1$ in section \ref{sect:AG}. Finally, we defined a 
new Hopf algebra object in section \ref{sect:tangent} related to the angle-addition formula 
for the function $\tan\theta$.
\end{abstract}
\maketitle

\section{Introduction and motivation}
Formulae such as the angle-addition laws of $\sin\theta$ and $\cos\theta$, the binomial 
theorem, and other such formulae involving the expression of $f(x+y)$ in terms of $f(x)$ 
and $f(y)$, for a function $f$, are often encapsulated by coalgebras\cite{represent}. 
In this paper, we consider the trigonometric coalgebra, which is the two dimensional 
coalgebra corresponding to the angle-addition formulae. We investigate certain properties of 
it which depend in an interesting way on the field over which the coalgebra is defined. 
Motivated by this investigation we define a $p$-dimensional coalgebra, which encapsulates 
the properties of circulant matrices and the finite Fourier transform in a similar manner. 

We then consider the Hopf algebra of trigonometric polynomials, i.e. the polynomial functions 
on the unit circle, considered as an algebraic variety. To conclude, we define a Hopf 
algebra-object (not formally a Hopf algebra, because it requires a completion of the 
tensor product) corresponding to the addition formula for $\tan\theta$, which we have not 
been able to find in the literature.

\section{The trigonometric coalgebra}\label{sect:trigco}
Throughout this paper, $k$ will represent a field. Consider the following trichotomy 
of cases for $k$:
    \begin{enumerate}[(a)]
        \item $\chr k=2$;
        \item $\chr k\neq 2$ and $-1$ has a square-root in $k$ (which we will write as 
        $\sqrt{-1}\in k$);
        \item $\chr k\neq 2$ and $\sqrt{-1}\notin k$.
    \end{enumerate}

In case (b) we will fix a root of $-1$ and write $\sqrt{-1}=i$, 
referring to the other as $-i$. If we wish to refer to cases (a) and (b) simultaneously, 
we will say ``if $\sqrt{-1}\in k$,'' acknowledging that $\sqrt{-1}=1$ in case (a).

Recall the definition of the \textbf{trigonometric coalgebra} $C_k$, which is defined over 
any base field (or ring): it is the 2-dimensional vector space with basis 
$\{c, s\}$, with coalgebra structure maps:
\begin{align*}
\Delta(c) &= c\otimes c - s\otimes s, & \epsilon(c) &= 1,\\
\Delta(s) &= s\otimes c+c\otimes s, & \epsilon(s) &= 0.
\end{align*}

This is clearly motivated by the angle-addition laws of $\sin\theta$ and $\cos\theta$: 
when $k=\R$, $C_\R$ is the \emph{representative coalgebra} associated to the 2-dimensional 
representation of the circle group $S^1$ as rotation matrices (for the more general 
constructions see \cite{represent}). The point of this interpretation is that $C_\R$ is 
the linear span of the matrix coefficients of this representation.

A natural question to ask is: is $C_k$ a bialgebra? It is certainly not a subalgebra of 
functions on $S^1$ in the case $k=\R$ or $k=\C$, and this should be remembered to avoid 
confusion.

\subsection{Bialgebra extension problem for $C_k$} By ``bialgebra extension problem'', 
we mean the question of whether a coalgebra can be extended to a bialgebra with 
the appropriate algebra structure morphisms, and if this bialgebra is unique (up to 
isomorphism).
\begin{theorem}\label{thm:extension_problem}
 The solution to the bialgebra extension problem for $C_k$ depends on $k$, according to the 
 trichotomy given above -- 
 \begin{enumerate}[C{a}se (a):]
  \item There are two non-isomorphic bialgebra structures on $C_k$, and 
  the set of valid structure constants for the algebra maps are in bijection with $k$.
  \item There is one isomorphism class of bialgebra which $C_k$ extends to, with 
  two possible sets of structure maps to define it.
  \item In this case $C_k$ is not a bialgebra.
 \end{enumerate}
\end{theorem}

\begin{proof}
 For a coalgebra $C$, we write $G(C)$ for the set of grouplike elements. We prove Theorem 
 \ref{thm:extension_problem} in the reverse direction, since this is the order of 
 increasing difficulty:
 \begin{enumerate}[(a)]
  \item[Case (c)] If $C_k$ is a bialgebra, the identity must be a grouplike element. Checking 
  the general element $\alpha c+\beta s\in C_k$, we see it is grouplike iff $\alpha=1$ and 
  $\beta^2=-1$. Thus in case (c) there are no grouplikes.
  \item[Case (b)] When $\sqrt{-1}\in k$, the prior calculation now shows that
  $$G(C_k)=\{c\pm is\}.$$
  Since $|G(C)|=2$ in case (b), we have a basis of $C_k$ consisting of grouplikes. Thus 
  $C_k$ is isomorphic as a coalgebra to $k[\Z/2]$. It is easy to show (see \cite{DNR} p. 169) 
  that in cases (b) and (c), any two-dimensional Hopf algebra is isomorphic to 
  $k[\Z/2]$. So, by transfer of structures, coalgebra isomorphisms $C_k\to k[\Z/2]$ are 
  in bijection with the possible Hopf algebra structure maps for $C_k$.

  Coalgebra isomorphisms are just given by bijections of the grouplike elements in this 
  case. Let us write $\ket 0, \ket 1$ for the elements $0,1\in\Z/2$, where $\Z/2$ is thought 
  of as a subset of $k[\Z/2]$ (this is to prevent notational overload and confusion: 
  e.g. $\ket 0$ is the multiplicative identity of $k[\Z/2]$). Then there are exactly two 
  isomorphisms $\phi, \phi': C_k\to k[\Z/2]$, defined by: 
  \begin{align*}
   \phi(c+is) &= \ket 0 & \phi'(c+is) &= \ket 1 \\
   \phi(c-is) &= \ket 1 & \phi'(c-is) &= \ket 0.
  \end{align*}
  The multiplication table associated to choosing $\phi$ is thus: 
  \begin{align*}
   (c\pm is)^2 &= c+is \\
   (c\pm is)(c\mp is) &= c-is,
  \end{align*}
  In the basis of $\{c,s\}$, the multiplication table for $\phi$ reads as:
  \begin{align*}
   c^2 &= c & cs &= 0 \\
   sc &= 0 & s^2 &= -is.
  \end{align*}

  The multiplication table of $\phi'$ corresponds to conjugating the right hand sides of the 
  above (i.e. $i\mapsto -i$); we think of this as the Galois group of $\C$ over $\R$ acting 
  on the structure constants, or acting on the isomorphisms $\{\phi, \phi'\}$. 
  We could not hope to give a better explanation of how to interpret this action 
  than \cite{qiaochu}. In any case, the antipode is the identity.

  \item[Case (a)] In this modular characteristic, we see $G(C)=\{c+s\}$ contains 
  one element, so it must be the identity element. Since $\epsilon(c)=1$, co-unitality 
  implies $c^2=c+Bs$ for some $B\in k$. Because $\ker\epsilon=\langle s\rangle$ is a 
  two-sided ideal, and any two-dimensional algebra is commutative, we altogether have (for 
  some $B,C,D \in k$): 
  \begin{align*}
   c^2 &= c+Bs & sc &= Cs \\
   cs &= Cs & s^2 &= Ds.
  \end{align*}
  Adding the second line to the first and last lines implies $B+C=0$ and $B+D=1$, thus we 
  may write
  \begin{align*}
   c^2 &= c+\lambda s & sc &= \lambda s \\
   cs &= \lambda s & s^2 &= (\lambda+1)s.
  \end{align*}

  To show this is a bialgebra, we must check the compatibility of structure maps, and make 
  sure that multiplication is actually associative. We will find both to be true for all 
  $\lambda\in k$. Once we do this, it will follow that $C_k$ is a Hopf algebra, since 
  by inspection the identity map serves as the antipode.

  It is known that the coalgebra structure maps being algebra homomorphisms, is equivalent to 
  the algebra structure maps being coalgebra homomorphisms. So it suffices to apply 
  $\epsilon$ and $\Delta$ to both sides of the multiplication table above, and check that 
  both maps are multiplicative.

  Since the above multiplication table is commutative, and trivially 
  $(c^2)c=c(c^2), (s^2)s=s(s^2)$, it suffices to check $(c^2)s=c(cs)$ and $(s^2)c=s(sc)$ to 
  verify associativity. We find the former two equal $\lambda^2 s$, while 
  the latter two equal $\lambda(\lambda+1)s$. We have now shown $C_k$ to be a Hopf algebra.

  Last we must prove our claim that this yields exactly two isomorphism classes. Let us 
  temporarily write $C_k(\lambda)$ for the bialgebra given above, as the isomorphism class 
  depends on $\lambda$. Notice that in $C_k(0)$, the elements $c$ and $s$ are orthogonal 
  idempotents. In particular, no element squares to 0 since $\chr k=2$. However, 
  in $C_k(1)$ we see $s^2=0$, showing $C_k(0)\not\cong C_k(1)$..

  This shows that $C_k(\lambda)$ has at least two isomorphism classes, as $\lambda$ 
  varies over $k$. We now give the isomorphism
  $\Phi: C_k(0)\to C_k(\lambda)$ for $\lambda\neq 1$:
  $$\Phi(c) = c+\frac\lambda{\lambda+1}s\hspace{.5mm}, 
  \hspace{2cm}
  \Phi(s) = \frac 1{\lambda+1}s.$$ 
  This shows that $C_k(\lambda)\cong C_k(0)$ for $\lambda\neq 1$. Thus the generic case 
  (meaning for all but finitely many $\lambda\in k$), is 
  the case in which $C_k(\lambda)$ has a complete set of orthogonal idempotents.
\end{enumerate}
\end{proof}

\begin{remark}
 It is known that there are three isomorphism classes of two-dimensional 
 Hopf algebras in characteristic 2 \cite{DNR}. One is $k[\Z/2]$, another is its dual, and 
 the third is a self-dual Hopf algebra. Here we have given a natural construction 
 of the two non-group algebras.
\end{remark}

\begin{remark}
 The computation in case (b) of $G(C_k)=\{c\pm is\}$ clearly brings to mind Euler's identity. 
 In section \ref{sect:AG}, we will view $C_k$ as a sub-coalgebra of $\mathcal O(S^1)$, 
 which explains this coincidence. This is because the grouplike elements of $k(G)$, 
 the commutative Hopf algebra of $k$-valued functions on a group $G$, are precisely the 
 characters of $G$.
\end{remark}

\section{Groups of units in group algebras}\label{sect:units}
Since $\dim C_k=2,$ the prevalence of the number 2 in the above trichotomy is to be 
expected. We wish to generalize the trigonometric coalgebra, and the above numerology, 
to dimension and characteristic $p$. We proceed along the following lines:
\begin{enumerate}[(1)]
 \item Working over $k=\C$, we calculate the group of units in the algebra 
 $\C[\Z/p]$, which leads us to an automorphism $F$ of $\C[\Z/p]$.
 \item We define $\ket 0, \ket 1, \ldots, \ket{p-1}$ the same way that we defined $\ket 0, 
 \ket 1$ above: they are the elements of $\Z/p$, when thought of as a subset of $k[\Z/p]$. 
 We look at the Hopf algebra structure coefficients, with respect to the new basis 
 $F\left(\ket 0\right), \ldots, F\left(\ket{p-1}\right)$, and find them 
 to be integers as well.
 \item We will reduce both structure coefficients modulo $p$, and we will show that the 
 two resulting Hopf algebras are non-isomorphic in characteristic $\chr k=p$, despite being 
 isomorphic (by construction) over $k=\C$.
\end{enumerate}

Let us write the Hopf algebra $\C[\Z/p]$ as $A_p$ here for brevity. The procedure outlined 
above was initially motivated by the following observation:
\begin{claim}
 The group of units, $A_2^\times$, of the algebra $A_2$ is the complement of the two 
 lines $\ell_1, \ell_2$ spanned by $\ket 0+\ket 1$ and $\ket 0-\ket 1$.

 Under the isomorphism $\phi: C_k\to k[\Z/2]$ constructed in Theorem 1, case (b), 
 we see $2c\mapsto\ket 0+\ket 1, 2is\mapsto\ket 0-\ket 1$. Thus these are the two lines 
 spanned by $c$ and $s$, under $\phi$.
\end{claim}

We now have a question which has an easy-to-answer generalization in $\Z/p$: 
what is the group of units of $A_p$? And does it lead to a coalgebra generalizing $C_\C$?

If $\alpha_0\ket 0+\alpha_1\ket 1+\ldots+\alpha_{p-1}\ket{p-1}\in A$ is a general element, 
and we wish to solve the equation 
$$(\alpha_0\ket 0+\alpha_1\ket 1+\ldots+\alpha_{p-1}\ket{p-1})\cdot (\beta_0\ket 0+\beta_1\ket 1+\ldots+\beta_{p-1}\ket{p-1}) = \ket 0,$$
this is written in matrix notation as 
$$\left[\begin{matrix}
\alpha_0 & \alpha_{p-1} & \ldots & \alpha_1 \\
\alpha_1 & \alpha_0 & \ldots & \alpha_2 \\
\ldots & \ldots & \ldots & \ldots \\
\alpha_{p-1} & \alpha_{p-2} & \ldots & \alpha_0
\end{matrix}\right]
\cdot
\left[\begin{matrix}\beta_0 \\ \beta_1 \\ \ldots \\ \beta_{p-1} \end{matrix}\right] 
= 
\left[\begin{matrix}1 \\ 0 \\ \ldots \\ 0 \end{matrix}\right].$$ 

The coefficient matrix is a circulant (i.e. the diagonals are constant). Letting $\omega$ be 
a nontrivial $p$\textsuperscript{th} root of unity, it is known that the eigenvalues of this 
circulant matrix are $\lambda_j=\sum_{i=0}^{p-1} a_i\omega^{ij}$ where $0\le j<p-1$. 
Each equation $\lambda_j=0$ defines a hyperplane $H_j\subset A$, orthogonal to the line 
spanned by $F_j=\frac 1p\sum_{i=0}^{p-1}\ket i\omega^{ij}\in A$ (under the inner product 
in which $\big\{\ket i\big\}_{i=0}^{p-1}$ forms an orthogonal basis in $A_p$). The 
complement of these $p$ hyperplanes is hence $A_p^\times$.

We will explain the normalization factor of $\frac 1p$ in the definition of $F_j$ in the 
next section. In summary: with $k=\C$ we have found a new basis 
$\big\{F_i\big\}_{i=0}^{p-1}$ of $A_p$, seemingly coming from $A_p^\times$. 
In the following section, we will write the (co)multiplication laws of $A_p$ in this new 
basis, and use this to define a Hopf algebra $C_{k,p}$ over any base field $k$. 

\begin{remark}
 If $\chr k=2$, then $\ell_1=\ell_2$ in $A_2$, and $A_2^\times$ is the complement of a line. 
 This coincides with the general result that if $\chr k=p$ and $G$ is a $p$-group, 
 then $k[G]^\times$ is the complement of the augmentation ideal $\ker\epsilon$ (which is 
 again a hyperplane), since this ideal coincides with the Jacobson radical of the ring.
\end{remark}

\section{Finite Fourier transforms}\label{sect:FFT}
\begin{prop}
 In $A_p$, the basis $\big\{F_i\big\}$ defined above has structure maps defined as 
 \begin{align*}
  \Delta(F_j) &= \sum_{I+J=j\text{ mod }p} F_I\otimes F_J 
  & 
  \epsilon(F_j) &= \delta_{0j}.
 \end{align*} 
\end{prop}

\begin{proof}
 Apply $\epsilon$ to the definition of $F_j$ above. As for $\Delta$, observe 
 $F_j=\frac 1p\sum_{k=0}^{p-1}\ket k \omega^{jk}$ and 
 $\sum_{k=0}^p F_k \omega^{-jk}=\ket j$. These are the formulae for the finite Fourier 
 transform and its inverse, for $\Z/p$!
 
 Apply $\Delta$ to the latter equation and insert this into the former. 
 The coefficients being geometric series will yield the proposed result.
\end{proof}

Observe also that the $F_i$ are orthogonal idempotents in $A_p$, and 
that the identity is $\ket 0=\sum_{i=0}^{p-1} F_i$. Finally, the antipode of $A_p$ being 
defined by $S\left(\ket i\right)=\ket{p-i}$ shows that $S(F_i)=F_{p-i}$. Thus all of the 
Hopf algebra structure maps have integer coefficients, in the basis $\{F_i\}$. However, we 
will show below that in characteristic $p$, the resulting coalgebra has other multiplication 
laws which would extend it to a Hopf algebra as well.

\begin{remark}
 This proposition explains the factor of $\frac 1p$ in our definition of $F_j$: 
 without it, the right hand side of the equations for $\epsilon$ and $\Delta$ 
 would be multiplied by $p$, and this would not yield a coalgebra 
 upon reducing coefficients modulo $p$.
\end{remark}
\begin{remark}
 The factor of $\frac 1p$ in the finite Fourier transform is analogous to the factor of 
 $\frac 1{2\pi}$ in the continuous Fourier transform. However, in the continuous version 
 one may choose, by convention, whether the $\frac 1{2\pi}$ factor appears in the Fourier 
 transform formula, or the inverse transform formula (or whether both contain a factor of 
 $\frac 1{\sqrt{2\pi}}$).
 For our purposes however, we must choose $\frac 1p$ to be in the finite Fourier transform.
\end{remark}

\subsection{The circulant coalgebras}
Again $k$ is an arbitrary field.
\begin{definition}
 Let $C_{k,p}$ be the $p$-dimensional coalgebra over $k$, defined by the basis 
 $\big\{F_i\big\}$, and by the formulae in the above proposition with 
 coefficients considered as elements of $k$. We call this the 
 \textbf{circulant coalgebra}, since it reflects the symmetry of the 
 circulant matrix, as demonstrated in the computations above.
\end{definition}

Let us here refer to a modified trichotomy of conditions on $k$:
\begin{enumerate}[(a')]
 \item $\chr k=p$;
 \item $\chr k\neq p$ and 1 has a nontrivial $p$\textsuperscript{th} root in $k$;
 \item $\chr k\neq p$ and 1 does not have a nontrivial $p$\textsuperscript{th} root in $k$.
\end{enumerate}

As we said, we can consider the structure coefficients of the above Hopf algebra 
as elements of $k$, giving us the \textbf{(standard) circulant Hopf algebra} $H_{k,p}$. 

In case (b') the circulant coalgebra is 
spanned by grouplike elements, like the trigonometric coalgebra was in case (b). Thus any 
Hopf algebra structure extending $C_{k,p}$ will be isomorphic to $H_{k,p}$, 
since $\Z/p$ is the only group of order $p$. There are $p!$ possible multiplication tables, 
for the same reasons as in case (c) above. 

In case (c'), $H_{k,p}$ is a Hopf algebra extending $C_{k,p}$, unlike the 
trigonometric coalgebra in case (c), which has no Hopf algebra extension: 
more on this after the following theorem. $H_{k,p}$ has only 1 grouplike element, the 
identity. We are unsure if there are other algebra structure maps which extend $C_{k,p}$ 
to a Hopf algebra in case (c'), or if there are multiple isomorphism classes of such Hopf 
algebras.

\begin{theorem}
 In case (a'), the set of possible structure constants for the algebra maps 
 which extend $C_{k,p}$ into a bialgebra are in bijection with $k$.
\end{theorem}

\begin{proof}
Recall that the dual vector space of a coalgebra always comes equipped with a natural algebra 
structure, and vice-versa in the finite-dimensional case. Let us use this intepretation: 
the coalgebra $C_{k,p}$ is the dual of the algebra 
$A_{k,p}=k[t]/\langle t^p-1\rangle=k[t]/\langle (t-1)^p\rangle$. 
Bialgebra structures on one are in bijection with bialgebra structures on the other by 
duality, and we find it simpler to work here with $A_{k,p}$.

The maps $\epsilon$ and $\Delta$ must be algebra homomorphisms. Thus we must choose 
$\epsilon(t)\in k, \Delta(t)\in A\otimes A$ to be $p^{\text{th}}$ 
roots of unity in their respective codomains, hence $\epsilon(t)=1$. 
Writing $A\otimes A\cong k[s,t]/\langle (s-1)^p, (t-1)^p\rangle$, we see the 
$p^{\text{th}}$ roots of unity are given by the affine hyperplane of polynomials whose 
coefficients sum to 1. If we write $\Delta(t)=f(s,t)$, then counitality is equivalent 
to saying $f(1,t)=f(t,1)=t$. So if $\Delta(t)=Ats+Bt+Cs+D$, we have the independent equations:
\begin{align*}
A+B+C+D &= 1 \\
B+D &= 0 \\
C+D &= 0.
\end{align*}
So then $\Delta(t)=(1-B)ts+Bt+Bs-B$, and the theorem follows from the rank of this 
system of linear equations.
\end{proof}

Let us call this bialgebra $C_{k,p}(B)$ to emphasize the dependence on the constant $B\in k$. 
We have not analyzed the question of if $C_{k,p}(B)$ has an antipode, except in the 
standard circulant Hopf algebra case above (which is $C_{k,p}(0)$).

We have also not analyzed how the isomorphism class of $C_{k,p}$ varies as the 
arbitrary structure constant $B\in k$ varies. However, referring to refer 
to \cite{classification}, we believe a similar pattern as with the trigonometric coalgebra 
will occur. Specifically their Theorem 2.1 shows that if $H$ is a $p$-dimensional 
Hopf algebra over an \emph{algebraically closed} field $k$ of characteristic $p$, 
there are exactly 3 possible isomorphism classes. (This is a modular version of the 
well-cited Kac-Zhu theorem \cite{kaczhu}, which states that when $k$ 
is algebraically closed and characteristic $0$, any Hopf algebra of dimension $p$ 
is isomorphic to $k[\Z/p]$.)

In the 3 cases in \cite{classification}, they show $H$ is generated by a single element; 
in the two cases which are not group algebras, one is generated by a nilpotent element 
(case 2 in \cite{classification}), and the other is generated by an idempotent element 
(case 3). We believe that the case (3) of that theorem will hold for generic $B$, 
and case (2) will hold for a finite number of special values, since this is exactly 
what occured for the trigonometric coalgebra in our modular case (a), above.

\begin{remark}
We saw in case (c'), there exist a Hopf algebra extending the coalgebra $C_{k,p}$. This 
contrasts with case (c) from Theorem \ref{thm:extension_problem}, where there were no 
grouplikes. In particular, $G(C_{k,p})=\{F_0+F_1+\ldots+F_{p-1}\}$ is nonempty.

Thus $C_k\not\cong C_{k,2}$ in cases (c, c'). Ultimately this traces back to the calculation 
in section 2 that $\ket 0+\ket 1\mapsto -2is$ under the coalgebra isomorphism 
$\phi: C_k\to k[\Z/2]$. This then traces back to the definition of comultiplication: 
in $C_{k,2}$ we defined $\Delta(F_0)=F_0\otimes F_0+F_1\otimes F_1$. Notice that 
$C_{k,2}$ reflects the addition formulae for the \emph{hyperbolic} trigonometric functions! 
We might then call $C_{k,2}$ the \textbf{hyperbolic (trigonometric) coalgebra}.

The fact that Euler's identity is defined as an equation over $\C$ and has non-real constant 
in it, yet the coefficients in the addition formulae of $\sin(\theta), \cos(\theta)$ are 
integers, can therefore be said to ultimately cause this discrepancy in the existence 
question for the bialgebra extension problems for $C_k$ vs. $C_{k,2}$, in cases (c) vs. (c'). 
The grouplikes in $C_{k,2}$ are $F_0\pm F_1$, which corresponds to the 
fact that $\cosh(x)\pm\sinh(x)=e^{\pm x}$, which is a homomorphism 
from $(\R,+)$ to $(\R,\cdot)$.
\end{remark}

\section{Trigonometric polynomials}\label{sect:AG}
Considering trigonometry from a co-algebraic standpoint, we should also look at the regular 
functions ring $\mathcal O\left(S^1\right)=k[c,s]/\langle c^2+s^2=1\rangle$, i.e. the ring 
of polynomials on the circle. This has the structure of a Hopf algebra since $S^1$ is an 
algebraic group variety: $\Delta(c), \Delta(s)$ are defined the same as in the 
trigonometric coalgebra, but multiplication is different. 
Since $\Delta(c)^2+\Delta(s)^2=1\in \mathcal O(S^1)\otimes\mathcal O(S^1)$, 
and similarly $\epsilon(c)^2+\epsilon(s)^2=1$, $\Delta$ and $\epsilon$ are 
well-defined algebra homomorphisms (using the universal mapping properties of polynomial 
algebras and of quotients).

We refer back to the trichotomy in \ref{thm:extension_problem}.
\subsection{Cases (b) and (c)}
In cases (b) and (c), we see that $\mathcal O(S^1)$ has a natural direct sum decomposition 
into sub-coalgebras. In case (b) we find 
\begin{equation}\label{eqn:decompositionB}
\mathcal O(S^1)=\bigoplus_{n\in\Z} C_n,
\end{equation}
where $C_n$ is the 1-dimensional span of $(c+is)^n$. Thus it is grouplike, and 
$\mathcal O(S^1)\cong k\left[\Z\right]$ as Hopf algebras.

In case (c), the decomposition becomes:
\begin{equation}\label{eqn:decompositionC}
 \mathcal O(S^1)=\bigoplus_{i=0}^\infty C_i,
\end{equation}
where $C_0$ is the span of the identity element, and $C_n$ is the span of $\left\{T_n(c), 
U_{n-1}(c)\cdot s\right\}$, where $T_n, U_n$ are the Chebyshev polynomials of the 
1\textsuperscript{st} and 2\textsuperscript{nd} kinds respectively.

Formulae \ref{eqn:decompositionB} and \ref{eqn:decompositionC} give decompositions 
into simple objects, showing that in cases (b) and (c), $\mathcal O(S^1)$ is semisimple as a 
coalgebra. In both formulae, the summands are the 
representative coalgebras of irreducible representations of the circle as a compact group: 
\ref{eqn:decompositionB} enumerates the complex-valued representations, and 
\ref{eqn:decompositionC} enumerates the real-valued ones.

Thus, these decompositions should be thought of as a polynomial version of the 
Peter-Weyl theorem for $S^1$.

\subsection{Case (a)}
In the modular characteristic, the above decompositions do not work. For 
\ref{eqn:decompositionB}, the only square root of $-1=1$ is $1$, and $(c+s)^2=1$. 
As for \ref{eqn:decompositionC}, the Chebyshev polynomials of both types have 
even coefficients in all degrees at least 2: so they are at most linear for all $n$, 
when reduced modulo 2. Furthermore, $\mathcal O(S^1)$ is no longer semisimple: 
$s^2$ is a primitive element, so the span of $\{1,s^2\}$ forms an indecomposable, 
but not irreducible, sub-coalgebra.

The most interesting question to us is whether $\mathcal O(S^1)$ is pointed as a coalgebra, 
since by many authors' account this is the definition of a quantum group. We saw in the 
prior subsection that in case (c), $\mathcal O(S^1)$ is not pointed, whereas in case (b) 
it is. We are presently unsure of the answer in case (a) in general: it is true for any 
field containing $\overline{\Z/2}$, the algebraic closure of $\Z/2$ as we next show, 
but we do not know if it is true for $k=\Z/2$.

One partial result in determining pointed-ness is from \cite{sweedler} (p. 158), 
whose proof we replicate, but whose conclusions we slightly expand:
\begin{theorem}\label{thm:sweedler}
 A simple subcoalgebra of a cocommutative coalgebra $C$ over the field $k$, is the 
 dual-coalgebra to a finite field extension $K\supset k$ (considered as a $k$-algebra).
\end{theorem}
\begin{remark}
 In \cite{sweedler}, this result was stated for algebraically closed $k$, in which case it 
 states that any cocommutative coalgebra must be pointed. Since pointed-ness is preserved 
 by extension of scalars, this shows that $\mathcal O(S^1)$ is pointed when 
 $k\supset\overline{\Z/2}$ as claimed.
\end{remark}
\begin{proof}
 Any simple coalgebra is finite-dimensional, and in the duality of finite dimensional 
 coalgebras $C$ with finite dimensional algebras $C^*$, it is established that ideals of 
 $C^*$ are in bijection with subcoalgebras of $C$. Thus $C^*$ is simple as a $k$-algebra, 
 i.e. it is a field, proving the claim.
\end{proof}

Thus if $\mathcal O(S^1)$ is not pointed over $k=\Z/2$, it contains a simple sub-coalgebra 
whose dimension is a power of 2.

We will now describe what we know of the structure of $\mathcal O(S^1)$.
\subsubsection{Square-roots of unity} 
Here we present our partial results on the grouplike elements of $\mathcal O(S^1)$. First, 
$\mathcal O(S^1)$ has grouplike elements 1 and $c+s$, and we are unsure if it has more, 
which is a main impediment in determining if it is pointed or not. 

Since the antipode of $\mathcal O(S^1)$ is the identity, any grouplike element must square 
to 1. If $x=\sum\big(a_i s^i+b_ics^i\big)$, we find by elementary calculations that 
$x^2=1$ if and only if $a_0+b_0=1$ and $a_i+b_i+b_{i-1}=0$ for all $i\ge 1$.

We may treat $a_0$ as a free variable, and let $b_0=1+a_0$, and inductively we may treat 
$a_i$ as a free variable, and $b_i=1+\sum_{j\le i}a_j$. Since this polynomial must 
terminate, $a_N=0$ for some smallest number $N$, and then $b_{N'}=b_N$ for all $N'\ge N$. 
Thus we require $b_N=0$ and thus $\sum_{i=0}^{N-1} a_i=1$. These conditions are necessary 
and sufficient for the polynomial to terminate and be an idempotent. We remain unsure of 
which idempotents, if any apart from 1 and $c+s$, are grouplikes.

\subsubsection{Primitives, and a formula for comultiplication}
We see that $cs$ and $s^2$ are primitive elements; since we are 
in characteristic 2, applying the Frobenius endomorphism, we know $s^{2^n}$ is primitive for 
all $n\ge 1$. So $(cs)^{2^n}$ is also primitive, but $(cs)^2=s^2+s^4$, so the higher powers 
are contained in the span of $\{s^{2^n}\}$. Thus $\{s^{2^n}\}_{n=1}^\infty\cup\{cs\}$ 
forms a linearly independent set of primitives. Again, we are unsure if the span of this set 
exhausts the set of all primitive elements.

Since $s^2$ is primitive, and not $c^2$, we find that $\{s^n\}\cup\{cs^n\}$ is a more 
natural basis of $\mathcal O(S^1)$ than $\{c^n\}\cup\{sc^n\}$ to work with.

This observation about primitives actually led us to the formula for the coproduct in this 
basis. First, let $[\ell]_2$ denote the binary expansion of $\ell$. 
Since $s^{2^n}$ is primitive, $\Delta(s^{2i})=\prod_k s^{2^k}\otimes 1+1\otimes s^{2^k}$, 
where $k$ ranges over the set of positions in $[2i]_2$ that contain 1 (we define 
the $2^i$'s digit as having position $i$). Expanding the product, we find 
$$\Delta(s^{2i})=\sum_{\ell, m} s^\ell\otimes s^m,$$ 
where $m=2i-\ell$, and $\ell$ ranges over the numbers such that $[\ell]_2$ has 1s only in 
positions that $[2i]_2$ has 1s (in computer science terms, this is the set of numbers such 
that the logical \textsc{and} of $\ell$ and $2i$ equals $\ell$). Then 
$\Delta(s^{2i+1})=(c\otimes s+s\otimes c)\sum_{\ell,m} s^\ell\otimes s^m$, and 
$\Delta(cs^{2i})=(c\otimes c+s\otimes s)\sum_{\ell,m} s^\ell\otimes s^m$. Finally, 
using the fact that $cs$ is also primitive, we arrive at the formula 
$$\Delta(cs^{2i+1})=cs^{2i+1}\otimes 1+1\otimes cs^{2i+1}+(s\otimes s)\Delta(c)
\sum_{\ell', m'}s^{\ell'}\otimes s^{m'},$$ 
where $\ell'$ ranges over the numbers such that $[\ell']_2$ has 1s only in the positions that 
$[2i-2]_2$ has 1s, and $m'=2i-2-\ell'$. This formula presents $\Delta(cs^{2i+1})$ as the 
sum of a ``primitive part'', and a summation of terms with smaller exponents in $s$.

If we let $V_1=\spn\{s^{2i}\}, V_2=\spn\big(\{cs^{2i+1}\}\cup\{1\}\big), 
W_1=\spn\{s^{2i+1}\}$, and $W_2=\spn\{cs^{2i}\}$, the above formulae show that 
$V_1$ and $V_2$ are sub-coalgebras of $\mathcal O(S^1)$, 
and that $\Delta(W_1), \Delta(W_2)\subset W_1\otimes W_2+W_2\otimes W_1$. We see that 
$V_1\cap V_2=\spn\{1\}$, and that all other intersections are trivial, so that 
$\mathcal O(S^1)=(V_1+V_2)\oplus(W_1\oplus W_2)$ is a decomposition of $\mathcal O(S^1)$ 
into two sub-coalgebras.

\subsection{The Chebyshev basis}
Possibly of future utility, we found a second basis interesting to work with. As noted 
above, the Chebyshev polynomials contain even coefficients in degrees at least 2. However, 
finding that the lead coefficient of $T_n$ is $2^n$ and the lead coefficient of 
$U_n$ is $2^{n-1}$, we considered the polynomials 
$$\widetilde{T_n}(x)=T_n\left(\frac x2\right) \hspace{1in} 
 \widetilde U_n(x)=2U_n\left(\frac x2\right)$$.

It follows from induction and the recurrence relations 
\begin{align*}
T_{n+1}(x) &= 2xT_n(x)-T_{n-1}(x) \\
U_{n+1}(x) &= 2xU_n(x)-U_{n-1}(x), \\
\end{align*} 
that $\widetilde{T}_n(x), \widetilde{U}_n(x)$ have integral coefficients as well. 
Furthermore, these polynomials are monic, and thus they form a basis.

\begin{conjecture}
 The modified Chebyshev polynomials satisfy $\widetilde T_n(x)=x^n$ iff $n=2^k$ for some $k$, 
 and $\widetilde U_n(x)=x^n$ iff $n=2^k-1$ for some $k$.
\end{conjecture}
One is able to prove the ``if'' direction of these claims using the fact that 
$\left\{\widetilde T_n(x), \widetilde U_n(x)\right\}_{n\in\N}$ form a so-called 
Lucas sequence. Specifically, one proves inductively that
$\widetilde U_{2n}(x)=\widetilde U_n^2(x)$ and $\widetilde T_{2n}(x)=\widetilde T_n(x) 
\widetilde U_n(x)$.

As for the ``only if'' direction: we have checked this up to $n=2^{30}+1$, 
using the above recursion relations, the GMP big-num(ber) library, and 
speeding up calculations by representing a degree $n$ polynomial over $\Z/2$ as a 
bitstring of length $n+1$. We were able to perform these calculations on a home laptop over 
the course of a night (this requires the computer have at least 6\textsc{GB} of accessible 
\textsc{RAM}). We have included the code in the online attached file ``\texttt{main.cpp}''.

With these observations and Theorem \ref{thm:sweedler} in mind, we conjecture that if 
$\mathcal O(S^1)$ is not pointed over $k=\Z/2$, there will be a simple subcoalgebra of the 
form $\left\{\widetilde T_n(c), \widetilde U_{n-1}(c)\cdot s\right\}$, 
where the index $n$ ranges either over $2^k\le n\le 2^{k+1}-1$, or $2^k-1\le n\le 2^{k+1}$. 

\section{The tangential Hopf algebra-object}\label{sect:tangent}
\subsection{Motivation}
Finally, we define a coalgebra/Hopf algebra which we have not seen in the literature, and 
which corresponds to another trigonometric identity. The tangent angle-addition law 
$$\tan(\alpha+\beta)=\frac{\tan\alpha+\tan\beta}{1-\tan\alpha\cdot\tan\beta}$$ looks like 
it should define a representative coalgebra, with the formula 
$\Delta(t)=\frac{1\otimes t+t\otimes 1}{1-t\otimes t}$. Expanding the bottom as a 
geometric series, we write 
$\Delta(t)=\sum_{n=0}^\infty t^{n+1}\otimes t^n+t^n\otimes t^{n+1}$. 

This indicates that we should consider the domain of $\Delta$, i.e. the coalgebra under 
construction, to be $k[[t]]$, and the codomain to be $k[[t]]\widehat\otimes k[[t]]$, where 
$\hat\otimes$ indicates that we take the $I$-adic completion of $k[[t]]\otimes_k k[[t]]$ 
with respect to the ideal $I=\langle t\otimes 1,1\otimes t\rangle$. This is because the 
formula for $\Delta(t)$ is not an element of $k[[t]]\otimes_k k[[t]]$. 

If we replace $t\mapsto ht$ on both sides of the formula for $\Delta(t)$ and cancel an 
overall factor of $h$ (this is valid since $k[t][[h]]$ is a domain), we get 
$\Delta(t)=\sum_{n=0}^\infty h^{2n}\left(t^{n+1}\otimes t^n+t^n\otimes t^{n+1}\right)$. 
This demonstrates our coalgebra as a 2\textsuperscript{nd} order deformation 
of the polynomial coalgebra $k[t]$.

\subsection{Definition}
This may be formalized: let $\mathcal C$ be the monoidal category of modules over 
$k[t][[h]]$, where the tensor bifunctor is the usual tensor product of $k$-modules, 
followed by $I$-adic completion along $I=\langle 1\otimes h,h\otimes 1\rangle$. 
We write this tensor followed by completion as $\widehat\otimes$.

The above calculations show that the unit object in this category $T=k[t][[h]]$ is a 
coalgebra-object. Defining $\epsilon(t)=0, S(t)=-t$, and using the usual 
algebra structure of $T$, we further see that $T$ is a Hopf algebra-object in 
$\mathcal C$.

We call $T$ the \textbf{tangential Hopf algebra (object)}. We choose to avoid the 
name \emph{tangent coalgebra}, in case the theory of Lie co-groups should develop to a 
point where this leads to a naming conflict.

\begin{remark}
 We may also consider the Hopf-object 
 $\Big(k[c,s,t]/\langle c^2+s^2=1, thc=s\rangle\Big)[[h]]$ in $\mathcal C$, where 
 $\Delta(t), \epsilon(t), S(t)$ is defined as above, $\Delta(c), \epsilon(c), S(c)$ is 
 defined as in $\mathcal O(S^1)$, and similarly for $s$, and call this the 
 \textbf{full trigonometric Hopf algebra (object)}.
\end{remark}

\begin{acknowledgements}
I would like to thank Prof. Susan Montgomery for entertaining my questions, and referring me 
to the paper \cite{classification} just as it emerged on the cutting-edge. And I would like 
to thank my advisor, Nicolai Reshetikhin for helping me to avoid the 
overly common ``death by planning" antipattern. Finally, Mario Sanchez for asking the 
question ``Is the trigonometric coalgebra an algebra as well?'' during seminar, which 
ultimately led me to address the questions in this paper.
\end{acknowledgements}

\end{document}